\documentclass[12pt]{amsart}

\newtheorem{theorem}{Theorem}[section]

\newtheorem{corollary}[theorem]{Corollary}

\theoremstyle{definition}

\newtheorem{definition}[theorem]{Definition}

\newtheorem{remark}[theorem]{Remark}

\newtheorem{example}[theorem]{Example}

\theoremstyle{parrafo}

\begin{document}

\title[]{The Stein Str\"omberg Covering Theorem in metric spaces}

\author{J. M. Aldaz}
\address{Instituto de Ciencias Matem\'aticas (CSIC-UAM-UC3M-UCM) and Departamento de 
Matem\'aticas,
Universidad  Aut\'onoma de Madrid, Cantoblanco 28049, Madrid, Spain.}
\email{jesus.munarriz@uam.es}
\email{jesus.munarriz@icmat.es}


\thanks{The author was partially supported by Grant MTM2015-65792-P of the
MINECO of Spain, and also by by ICMAT Severo Ochoa project SEV-2015-0554 (MINECO)}






\begin{abstract} In \cite{NaTa} Naor and Tao extended to the metric setting the $O(d \log d)$ 
bounds given by Stein and Str\"omberg for Lebesgue measure in 
$\mathbb{R}^d$, deriving these bounds first from a localization result, and second, from  a
random Vitali lemma. Here we show that the Stein-Str\"omberg original argument
can also be adapted to the metric setting, giving a third proof. 
We also weaken the hypotheses, and 
additionally, we sharpen the
estimates for Lebesgue measure.
\end{abstract}


\maketitle


\markboth{J. M. Aldaz}{Stein  Str\"omberg covering theorem}

\section {Introduction} 

In \cite{StSt}, Stein and Str\"omberg proved that for Lebesgue measure in $\mathbb{R}^d$,
and with balls defined by an arbitrary norm, the centered maximal function has weak type (1,1)
bounds of order $O(d \log d)$, which is much better than the exponential bounds obtained via
the Vitali covering lemma.  Naor and Tao   extended  the  Stein-Str\"omberg
 result to the metric setting in  \cite{NaTa}.
 There, a localization result is proven (using the notion of  microdoubling, which basically entails a
 very regular growth of balls) from which the  Stein-Str\"omberg 
 bounds
 are obtained (using  the notion of  strong microdoubling, which combines microdoubling with
 local comparability). 
Also, a second argument is given, via a random Vitali Theorem that has  its
 origin in \cite{Li}.

 Here we note that the Stein-Str\"omberg original proof,
 which is shorter and conceptually simpler, can also be used in the metric setting,
 yielding a slightly more general result.
 We will
divide the  Stein-Str\"omberg argument  into two parts, one with radii separated by large gaps, and the second, with radii inside
an interval, bounded away from 0 and $\infty$. This will allow us to obtain more precise 
information about which hypotheses are needed in each case. We shall see that under the same hypotheses
used by Naor and Tao, the  Stein-Str\"omberg covering theorem for sparse radii  (cf. Theorem \ref{StSt1} below)
suffices to obtain the $d\log d$ bounds in the metric setting. But Theorem \ref{StSt1} itself is presented
in a more general version. In particular, one does not need to assume that the metric space is
geometrically doubling.

We also show that the Stein-Str\"omberg method, applied to balls with no restriction in the  radii, yields the  $O(d \log d)$ 
bounds in the metric context, for doubling measures where the growth of balls can be more 
irregular than is allowed by the microdoubling condition. 
Finally, we lower the known weak type (1,1) bounds in the case of Lebesgue measure:
For $d$ lacunary sets of radii, from $(e^2 + 1) (e + 1)$ to 
$(e^{1/d} + 1) (1 + 2  e^{1/d})$ (to 6 in the specific case of $\ell_\infty$ balls),
and for unrestricted radii, from
 $e^2 (e^2 + 1) (1 + o(1)) d \log d$ to $(2 +  3 \varepsilon) d \log d$, where   $\varepsilon > 0$
 and  $d = d(\varepsilon)$ is sufficiently large.

\section {Notation and background material}

Some of the definitions here  come from 
\cite{A2}; we refer the interested reader to that paper for motivation and additional explanations.

We will use $B(x,r) := \{y\in X: d(x,y) < r\}$ to denote open balls, 
$\overline{B(x,r)}$ to denote their topological closure, and 
$B^{cl}(x,r) := \{y\in X: d(x,y) \le r\}$ to refer to closed balls.
Recall that in a general metric space,  a ball $B$, considered as a set, can have many centers and many radii. When
we write $B(x,r)$ we mean to single out $x$ and $r$, speaking respectively of the center and the radius 
of $B(x,r)$.

\begin{definition} A Borel measure is {\em $\tau$-smooth} if for every
collection  $\{U_\alpha : \alpha \in \Lambda\}$
 of  open sets, $\mu (\cup_\alpha U_\alpha) = \sup \mu(\cup_{i=1}^nU_{\alpha_i})$,
 where the supremum is taken over all finite subcollections of $\{U_\alpha : \alpha \in \Lambda\}$.
 We say that $(X, d, \mu)$ is a {\em metric measure space} if
$\mu$ is a Borel measure on the metric space $(X, d)$, such that for all
balls $B(x,r)$, $\mu (B(x,r)) < \infty$, and furthermore, $\mu$ is $\tau$-smooth. 
\end{definition} 

The assumption of $\tau$-smoothness does not represent any real restriction, since it is consistent with standard set theory
(Zermelo-Fraenkel with Choice) that  in every metric space, every Borel measure  which
assigns finite measure to balls  is $\tau$-smooth 
(cf. \cite[Theorem (a), pg. 59]{Fre}).  

\begin{definition}\label{maxfun} Let $(X, d, \mu)$ be a metric measure space and let $g$ be  a locally integrable function 
on $X$. For each  $x\in X$, the centered Hardy-Littlewood maximal operator $M_{\mu}$ 
 is given by
\begin{equation}\label{HLMFc}
M_{\mu} g(x) := \sup _{\{r : 0 < \mu (B(x, r))\}}  \frac{1}{\mu
(B(x, r))} \int _{B(x, r)}  |g| d\mu.
\end{equation}
\end{definition}

Maximal operators can be defined using closed balls instead of open balls,
and this does not change their values, as can be seen by an approximation argument.
 When
the measure is understood, we will omit the subscript $\mu$  from  $M_\mu$.

A sublinear operator $T$ satisfies a
weak type $(1,1)$ inequality if there exists a constant $c > 0$ such that
\begin{equation}\label{weaktype}
\mu (\{T g > s \}) \le \frac{c \|g\|_{L^1(\mu)}}{s },
\end{equation}
where $c=c(T,  \mu)$ depends neither on $g\in L^1 (\mu)$
nor on $s > 0$. The lowest constant $c$ that satisfies the preceding
inequality is denoted by $\|T\|_{L^1\to L^{1, \infty}}$.

\begin{definition} A Borel measure $\mu$ on $(X,d)$ is {\em doubling}  if there exists a 
$C> 0 $ such that for all $r>0 $ and all $x\in X$, $\mu (B(x, 2 r)) \le C\mu(B(x,r)) < \infty$. 
\end{definition}

\begin{definition} \label{geomdoub} A metric space is {\it $D$-geometrically doubling}  if there exists a positive
integer $D$ such that every ball of radius $r$ can be covered with no more than $D$ balls
of radius $r/2$. 
\end{definition}

If a metric space supports a doubling measure, then it is geometrically doubling.
Regarding weak type inequalities for the maximal operator, in order to estimate $\mu \{M f > s\}$, one considers balls $B(x,r)$ over which 
$|f|$ has average
larger than $s$. Now, while in the uncentered case any such ball is contained in the corresponding
level set, this is not so for the centered maximal function. Thus, using the balls $B(x,r)$ to cover
$\{M f > s\}$ can be very inefficient. 

A key ingredient in the Stein-Str\"omberg proof is to cover
$\{M f > s\}$ by the much smaller balls $B(x,t r)$, $0 < t <<1$, something that leads to the
``microdoubling" notion of Naor and Tao. 
We slightly modify their notation, using $1/n$-microdoubling
to denote what these authors call  $n$-microdoubling.

 \begin{definition} (\cite[p. 735] {NaTa} Let $0 < t < 1$ and let $K\ge 1$. Then 
  $\mu$ is said to be $t$-microdoubling with constant $K$ if for all $x \in X$ and all $r >0$,
we have
$$
\mu B\left(x,\left(1 + t \right) r \right) \le K \mu B(x,r).
$$
\end{definition}

The next property is mentioned in \cite{NaTa}, and more extensively studied
in \cite {A2}. 

\begin{definition} \label{loccomp}  A  measure $\mu$ satisfies a {\it local comparability condition} 
if there
exists a constant $C\in[1, \infty)$ such that for  all pairs of points $x,y\in X$, and all $r >0$,
whenever $d(x,y) < r$, we have 
$$\mu(B(x,r))\le C \mu(B(y,r)).$$
 We denote the
smallest such $C$ by $C(\mu)$ or $C_\mu$.
\end{definition}

\begin{remark} \label{doublingandmicro}
 If $\mu$ is doubling with constant $K_1$ then it is microdoubling  and satisfies a local comparability
condition with the same constant $K_1$, while
if it is  $t$-microdoubling with constant $K_2$ and $2 \le (1 + t)^M$, then $\mu$ is doubling
and satisfies a local comparability
condition
with constant   $K_2^M$. Thus, the difference between doubling and 
microdoubling lies in the size of the constants, it is quantitative, not qualitative: The microdoubling condition adds something new only 
when $K_2 < K_1$, in which case it entails a greater regularity in the growth of the measure
of balls, as the radii increase. Likewise, 
bounds of the form $\mu B(x, T r) \le K \mu B(x, r)$ for $T > 2$, allow a greater irregularity in the growth of balls
than standard doubling ($T = 2$) or than microdoubling. 

We mention that  while local comparability is implied by
doubling, it is a uniformity condition, not a growth condition. Thus, it is compatible with
the failure of doubling, and even for doubling measures, it is compatible with
any rate of growth for the volume of balls.
Consider, for instance, the case of $d$-dimensional Lebesgue measure $\lambda^d$:
A doubling constant is $2^d$, a $1/d$-microdoubling constant is $e$,  and the smallest 
local comparability constant is $C(\lambda^d) = 1$. 

\end{remark}

The next definition combines the requirement that the microdoubling and the local comparability
constants be ``small" simultaneously. 

 \begin{definition} (\cite[p. 737] {NaTa} Let $0 < t < 1$ and let $K\ge 1$. Then 
  $\mu$ is said to be strong $t$-microdoubling with constant $K$ if for all $x \in X$, 
   all $r > 0$, and all $y\in B(x,r)$,
$$
\mu B\left(y,\left(1 +  t \right)r\right) \le K \mu B(x,r).
$$
\end{definition}

Thus, if $\mu$ is strong $t$-microdoubling with constant $K$, then $C(\mu) \le K$.
Also, local comparability is the same as strong $0$-microdoubling. To get a
 better understanding of how bounds depend on the different constants,
 it is useful to keep separate $C(\mu)$ and $ K$.

\begin{definition} Given a set $S$ we define its $s$-{\em blossom} as the enlarged set
\begin{equation} \label{altblossom}
Bl(S, s):= \cup_{x\in S}B(x,s),
 \end{equation}
 and its  {\em uncentered $s$-blossom} as the set
\begin{equation} \label{altublossom}
Blu(S, s):= \cup_{x\in S}\cup\{B(y, s): x\in B(y, s)\}.
 \end{equation}
When $S= B(x,r)$, we simplify the notation and write  $Bl(x,r, s)$, 
instead of $Bl(B(x,r), s)$, and likewise for uncentered blossoms. 
 We say that $\mu$ {\em blossoms boundedly} if there exists a $K\ge 1$ such that
 for all $r>0 $ and all $x\in X$, $\mu (Blu(x, r, r)) \le K \mu(B(x,r)) < \infty$. 
\end{definition}

Blossoms can be defined using closed instead of open balls, in
an entirely analogous way. To help understand the relationship between blossoms and balls,
we include the following definitions and results.

\begin{definition} A metric space  has the {\it approximate midpoint property} if for every
$\varepsilon > 0$ and every pair of points $x,y$, there exists a point $z$ such
that $d(x,z), d(z,y) <  \varepsilon + d(x,y)/2$.
\end{definition}

\begin{definition} \label{quasi}A metric space $X$  is  {\it quasiconvex} if there exists a constant
$C\ge 1$ such that for every
 pair of points $x,y$, there exists a curve with $x$ and $y$ as endpoints, such
 that its length is bounded above by $C d(x,y)$. If for every $\varepsilon > 0$ we can take  $C=1 + \varepsilon$, then we say that $X$ is a {\it length space}.
 \end{definition} 

It is well known that for a complete metric space, having the approximate midpoint property is equivalent
to being a length space.

\begin{example} \label{blossomsballs} The $s$-blossom of an $r$-ball may fail to contain a strictly larger ball,
even in quasiconvex spaces. 

For instance, let $X \subset  \mathbb{R}^2$ be the set $\{0\} \times [0,1]  \cup [0,1]\times \{0\}$ with metric
defined by restriction of the $\ell_\infty$ norm; then we can take $C = 2$. Now $B((1,0), 1) = (0,1]\times \{0\}$,
while for every $r > 1$,  $B((1,0), r) = X$, which is not contained in  $Blu((1,0), 1, 1/6)$. Furthermore, neither
$Blu((1,0), 1, 1/6)$ nor $Bl((1,0), 1, 1/6)$ are balls, i.e., given any $x\in X$ and any $r > 0$, we have
that $B(x, r) \ne Blu((1,0), 1, 1/6)$ and $B(x, r) \ne Bl((1,0), 1, 1/6)$.

On the other hand, if a  metric space $X$  has the approximate midpoint property,  then blossoms and
 balls coincide  (as we show next)
so in this case  considering
 blossoms gives nothing new. 
\end{example}

 \begin{theorem} \label{equiv}  Let $(X, d)$ be a metric space. The following are equivalent:

a)  $X$ has the approximate midpoint property.

b) For all $x\in X$, and all $r, s >0$, 
$Bl(x, r, s)  =  B(x, r + s).$

c) For all $x\in X$, and all $r >0$, 
$Bl(x, r, r)  =  B(x, 2 r).$
\end{theorem}

\begin{proof} Suppose first that  $X$ has the approximate midpoint property.
Since $Bl(x, r, s)  \subset  B(x, r + s),$ to prove b) it is enough to show that if
$y\in  B(x, r +  s),$ then $y \in  Bl(x, r, s)$, or equivalently, that
there is a $z\in X$ such that $d(x, z) < r$ and $d(z, y) < s$. 
If  either $d(x, y) < s$ or  $d(x, y) < r$ we can take $z = x$ and there is nothing to prove, so assume otherwise.
Let  $(\hat X,  \hat d)$ be the completion of  $(X, d)$; then $\hat X$ is a length space, since it
has the approximate midpoint property.  Let $\Gamma :[0,1] \to \hat X$ be a curve with $\Gamma (0) = x$, 
$\Gamma (1) = y$, and length $\ell (\Gamma) < r + s$. Then $\Gamma([0, 1]) \subset B(x, r) \cup B(y, s)$, 
for if there is a $w \in [0,1]$ with $\Gamma (w) \notin  B(x, r) \cup B(y, s)$, then $\ell (\Gamma)  \ge r + s$. 
Now let $c \in [0,1]$ be the time of first exit of $\Gamma (t)$ from $B(x,r)$, that is, for all $t < c$, 
 $\Gamma (t) \in B(x,r)$ and $\Gamma (c) \notin B(x,r)$ . Then $\Gamma (c) \in B(y,s)$, so by continuity of $\Gamma$, 
there is
a $\delta \in [0, c) $ such that  $\Gamma (\delta) \in B(y,s)$. Thus, the open set 
$ B(x, r) \cap B(y, s) \ne \emptyset$ in $\hat X$. But $X$ is dense in $\hat X$, so 
there exists a $z\in X$ such that $d(x, z) < r$ and $d(z, y) < s$, as we wanted to show.

Part c) is a special case of part b). From part c) we obtain a) as follows. Let $x, y \in X$, and let $r >0$ be such that $d(x,y) < 2 r$. By
hypothesis, $y\in Bl(x, r, r) = B(x, 2r)$, so there is a $z\in X$ such that $d(x, z) < r$ and  $d(z, y) < r$.
Thus, $X$ has the approximate midpoint property.
\end{proof}

\begin{example} \label{chordal} Let $X$ be the unit sphere (unit circumference) in the plane, with the chordal metric,
that is, with the restriction to $X$ of the euclidean metric in the plane. While this space does not
have the approximate midpoint property, blossoms are nevertheless geodesic balls.
However, the equality 
$Bl(x, r, s)  =  B(x, r + s)$ no loger holds. For instance,
 $Bl((1, 0), 1, 1) \ne Bl((1, 0), \sqrt2, \sqrt2 )  = B((1, 0), 2) = X \setminus \{(-1,0)\}$.
 \end{example}

\section{Microblossoming and related conditions}

 \begin{definition} \label{bmicroblu} Let $0 < t <  1$ and let $K\ge 1$. Then 
  $\mu$ is said to $t$-microblossom boundedly   with constant $K$,
   if for all $x \in X$ and all $r >0$,
we have
\begin{equation}
\mu (Blu \left(x, r,  t r \right)) \le K \mu B(x,r).
\end{equation}
\end{definition}

We shall say $\mu$  is a measure that $(t,K)$-microblossoms, instead of using the longer
expression.

 \begin{example} \label{bmicrobluex} 
  Microblossoming (even together with doubling) is  more general than microdoubling,
 in a quantitative sense. Consider  $(\mathbb{Z}^d, \ell_\infty, \mu)$, where $\mu$ is
 the counting measure. Then $\mu$ is doubling, and ``microdoubling in the large", since for
 large radii ($r > d$), $\mu$ can be regarded as a discrete approximation to Lebesgue measure.
 However,
 $\mu B(0,1) = 1$, and for every $t > 0$, $\mu B(0,1 + t ) \ge 3^d$, no matter how small
 $t$ is. Thus,  the measure $\mu$ is not $(t,K)$-microdoubling, for any $K < 3^d$, $0 < t << 1$. However,
 $\mu$ is $1/d$-microblossoming, since for $r > d$, $\mu$ behaves as a microdoubling 
 measure, and for $r \le d$, $Blu(x, r, r/d) = B(x,r)$.
 
 A less natural but stronger example is furnished by the measure $\mu$ given by
 \cite[Theorem 5.9]{A2}. Since $\mu$ satisfies a local comparability condition,
 and is defined in a geometrically doubling space, it blossoms boundedly, so it
 microblossoms boundedly (at least with the blossoming constant). But $\mu$  is not doubling,
 and hence it is not microdoubling. 
 \end{example}
 
 \begin{example} \label{bmicroblunonblu} While $(t,K_1)$-microdoubling entails
 $(2, K_2)$-doubling for some $K_2 \ge K_1$, the analogous statement is not true
 for microblossoming.  The following  example shows that $(1/2,1)$-microblossoming
 does not entail local comparability.
 Let $X = \{0, 1, 3\}$ with the inherited metric from $\mathbb{R}$, and let $\mu = \delta_3$.
 Then $B(0,3) \cap B(3, 3) = \{1\}$, but $\mu B(0,3) = 0$ while $\mu B(3,3) = 1$, so 
 local comparability fails. Since bounded blossoming
 implies local comparability, all we have to do is to check that $\mu$ is $(1/2,1)$-microblossoming.
 For $t\le 3$, $B(0,t) \subset Blu(0, t, t/2) \subset \{0,1\}$, so  $\mu B(0,t) = \mu Blu(0, t, t/2)= 0$,
 and for $t > 3$,  $B(0,t) =  Blu(0, t, t/2) = X$. Likewise, for $t\le 2$, $B(1,t) = Blu(1, t, t/2) \subset \{0,1\}$, so  
$\mu B(1,t) = \mu Blu(1, t, t/2)= 0$,
 and for $t > 2$,  $B(1,t) =  Blu(1, t, t/2) = X$.
 \end{example}

\begin{definition} Given a metric measure space $(X, d, \mu)$, and denoting the support of $\mu$ by 
$supp(\mu)$,
the {\em relative increment function} of $\mu$, $ri_{\mu}: supp(\mu)\times (0,\infty)\times [1,\infty)$,
is defined as
\begin{equation} \label{ri}
ri_{\mu}(x, r, t) := \frac{\mu B(x, tr)}{\mu B(x, r)},
\end{equation}
and the {\em maximal relative increment function}, as
\begin{equation} \label{mri}
mri_{\mu}(r, t) := \sup_{x \in supp(\mu)}\frac{\mu B(x, tr)}{\mu B(x, r)}.
\end{equation}
When $\mu$ is understood we will simply write $ri$ and $mri$.
\end{definition}

This notation allows one to unify different conditions that have been considered regarding
the boundedness of maximal operators. For instance, on $supp(\mu)$ 
the doubling condition simply
means that there is a constant $C\ge 1$ such that for all $r > 0$, $mri_{\mu}(r, 2) \le C$, and
the $d^{-1}$-microdoubling condition,
that for all $r > 0$, $mri_{\mu}(r, d^{-1}) \le C$. Note that by $\tau$-smoothness, the complement of
the support of $\mu$ has $\mu$-measure zero, so the relative increment function is defined
for almost every $x$. 

 \begin{example} \label{macrodoub}
 The interest of considering values of $t >2$ in the preceding 
 definition comes from the fact that, under the additional assumption of microblossoming, it
 will allow a  much  more irregular 
  growth of balls than microdoubling or plain doubling, without a comparable  worsening 
  of the estimates for the weak type (1,1) bounds. 

To fix ideas, consider the right hand side 
$ C(\mu) \ K_1 K \left(2 + \frac{\log K_2}{\log K}\right)$ 
of formula (\ref{sum}) below. This bound is related to the centered maximal operator when the supremum is restricted
to radii $R$ between $r$ and $T r$, $T > 1$. The constant $K_2$ depends on $T$, as it must satisfy 
$ mri_{\mu}(r, T) \le K_2$.
 For Lebesgue measure
on $\mathbb{R}^d$ with the $\ell_\infty$-norm, $ C(\lambda^d) = 1$. If we set $T=2$, then we can take
$K_2 = 2^d$, while $K_2 = d^d$ for $T=d$,   a choice which yields bounds of order $O( d \log d)$.
 A  $1/d$-microdoubling constant is $K_1 = e$ ($\mathbb{R}^d$ has the approximate midpoint
property, and in fact it is a geodesic space, so microdoubling is the same as microblossoming in this case) and $K := \max\{K_1, e\} = e$.
  
  Returning to Example \ref{bmicrobluex},  by a rescaling  argument it is clear that  the situation
  for 
    $(\mathbb{Z}^d, \ell_\infty, \mu)$ cannot be much worse than for
    $(\mathbb{R}^d, \ell_\infty, \lambda^d)$, and in fact it is easy to see that
    the same argument of Stein and Str\"omberg (which will be presented in greater
    generality below) yields  the  $O( d \log d)$ bounds. Now suppose  we modify
    the measure so that at one single point it is much smaller. Clearly, this will have
    little impact in the weak type (1,1) bounds, since for $d >>1$, $x \in \mathbb{Z}^d,$ and $r>1$, the measure of $B(x,r)$
    will be changed by little or not at all, while for $ r \le 1$, balls with distinct centers do not
    intersect. For definiteness, set $\nu = \mu$ on  $\mathbb{Z}^d \setminus \{0\}$,
    and $\nu \{0\} = d ^{-d}$. Then the doubling constant, and the 
$(t,K)$-microdoubling constant, for any $t > 0$, is at least $d ^{d} (3 ^{d} -1) \le  K = K_2$,  much larger than the corresponding constants for $\mu$.
However, the local comparability constant is still very close to 1, since intersecting balls of the same radius must
contain at least $3^d$ points each, and a $1/d$-microblossoming constant can be taken to be very close to $e$.
Setting 
$T = d$, we get $K_2 \le d^d (2d + 1)^d$, so $\log K_2$ in this case is comparable to the constant obtained when $T = 2$.
 \end{example}

 \begin{remark} One might define $(T, K)$-macroblossoming, 
 with $T > 1$, by analogy with Definition \ref{bmicroblu}. 
 However, since $B(x, T r) \subset Blu(x, r, T r)$, assuming
 directly that $mri_{\mu}(r, T) \le K$ is not stronger than $(T, K)$-macroblossoming, 
\end{remark}

\section{The Stein-Str\"omberg covering theorem}

Next, we present the Stein-Str\"omberg argument 
using the terminology of blossoms. Note  that
the next theorem does not require $X$ to be geometrically doubling.

Given an ordered sequence of sets $A_1, A_2, \dots$, we denote by $D_1, D_2, \dots$
its sequence of disjointifications, that is 
$D_1 = A_1$, and $D_{n + 1} = A_{n + 1}\setminus \cup_1^n A_i$. We shall avoid reorderings and
relabelings of collections of balls, as this may lead to confusion regarding the meaning of $D_j$. The
unfortunate downside of this choice  is an inflation of subindices.

\begin{theorem}\label{StSt1} {\bf Stein-Str\"omberg covering theorem for sparse radii.} Let $(X, d, \mu)$ be a metric measure space, where $\mu$ satisfies a $C(\mu)$ local comparability condition, and let $R:= \{r_n: n\in \mathbb{Z}\}$ be a $T$-lacunary sequence of radii, i.e.,
$r_n >0$ and $r_{n+1}/r_n \ge T > 1$. Suppose there exists a $t > 0$ such that $T  t \ge 1$ and 
$\mu$ $t$-microblossoms boundedly 
  with constant $K$. Let $\{B(x_i, s_i): s_i \in R, 1 \le i \le M\}$ be a finite collection
of balls with positive measure, ordered by non-increasing radii. Set $U:= \cup_{i = 1}^M   B(x_i, t s_i)$. Then there exists a
subcollection $\{B(x_{i_1}, s_{i_1}), \dots, B(x_{i_N}, s_{i_N})\}$, 
 such that, denoting by $D_{i_j}$  the disjointifications of the reduced balls $B(x_{i_j}, t s_{i_j})$,   
\begin{equation}\label{set}
\mu U \le (K + 1) \mu \cup_{j=1}^N   B(x_{i_j}, t s_{i_j}),
\end{equation}
 and 
 \begin{equation}\label{bound}
\sum_{j=1}^N \frac{\mu D_{i_j}}{\mu B(x_{i_j}, s_{i_j})}\mathbf{1}_{B(x_{i_j}, s_{i_j})}
\le C(\mu) \ K + 1.
\end{equation}
 \end{theorem}
 
 \begin{proof} We use the Stein-Str\"omberg selection algorithm. 
 Let $B(x_{i_1}, s_{i_1}) = B(x_{1}, s_{1})$
 and suppose that the balls $B(x_{i_1}, s_{i_1}), \dots ,B(x_{i_k}, s_{i_k})$ have already been selected. If
 $$ \sum_{j=1}^k \frac{\mu D_{i_j}}{\mu B(x_{i_j}, s_{i_j})}\mathbf{1}_{Bl(x_{i_j}, s_{i_j}, t s_{i_j})} (x_{i_{k} + 1}) \le 1,
$$
accept 
$B(x_{i_{k + 1}}, s_{i_{k + 1}}) := B(x_{i_{k} + 1}, s_{i_{k} + 1})$ as the next ball in the subcollection. Otherwise, reject it. Repeat till we run out of balls.
Let $\mathcal{C}$ be the collection of all rejected balls. Then $\mu$ a.e., 
$$
\mathbf{1}_{\cup \mathcal{C}} < \sum_{j=1}^N \frac{\mu D_{i_j}}{\mu B(x_{i_j}, s_{i_j})}\mathbf{1}_{Blu(x_{i_j}, s_{i_j}, t  s_{i_j})}.
$$
Integrating both sides and using microblossoming we conclude that
$\mu \cup \mathcal{C} \le K \sum_i^N \mu D_{i_j} =  K  \mu \cup_{j=1}^N   B(x_{i_j}, t s_{i_j})$,
whence $\mu U  \le (K + 1)  \ \mu \cup_{j=1}^N   B(x_{i_j}, t s_{i_j})$.

Next we show that $$
\sum_{j=1}^N \frac{\mu D_{i_j}}{\mu B(x_{i_j}, s_{i_j})}\mathbf{1}_{B(x_{i_j}, s_{i_j})}
\le C(\mu) \ K + 1.
$$
Suppose $\sum_{j=1}^N \frac{\mu D_{i_j}}{\mu B(x_{i_j}, s_{i_j})}\mathbf{1}_{B(x_{i_j}, s_{i_j})}(z) > 0$.
Let
$\{B(x_{i_{k_1}}, s_{i_{k_1}}),$ $ \dots, B(x_{i_{k_n}}, s_{i_{k_n}})\}$ be the collection of  all balls containing $z$ (keeping the original ordering by decreasing radii). 
Then each $B(x_{i_{k_j}}, s_{i_{k_j}})$
has radius either
equal to or (substantially) larger than $s_{i_{k_n}}$. We separate the contributions of these balls into two sums.  Suppose $B(x_{i_{k_1}}, s_{i_{k_1}}), \dots , B(x_{i_{k_m}}, s_{i_{k_m}})$ all have radii larger than $s_{i_{k_n}}$, while $B(x_{i_{k_{m} + 1}}, s_{i_{k_{m} + 1}}), \dots , B(x_{i_{k_n}}, s_{i_{k_n}})$
have radii equal to $s_{i_{k_{n}}}$. Now for $1 \le j \le m$, by $T$ lacunarity and the fact that $T t \ge 1$, we have
$s_{i_{k_n}} \le t s_{i_{k_j}}$, so $z\in B(x_{i_{k_j}}, s_{i_{k_j}})$ implies that 
$x_{i_{k_n}}\in Bl(x_{i_{k_j}}, s_{i_{k_j}}, t s_{i_{k_j}})$,
whence 
$$ 
\sum_{j=1}^m \frac{\mu D_{i_{k_j}}}{\mu B(x_{i_{k_j}}, s_{i_{k_j}})}
\mathbf{1}_{Bl(x_{i_{k_j}}, s_{i_{k_j}}, t s_{i_{k_j}})} (x_{i_{k_n}}) \le 1,
$$
and thus  
$$ 
\sum_{j=1}^m \frac{\mu D_{i_{k_j}}}{\mu B(x_{i_{k_j}}, s_{i_{k_j}})}
\mathbf{1}_{B(x_{i_{k_j}}, s_{i_{k_j}})}  (z) \le 1.
$$
Next, note that the sets $D_{i_{k_{m} + 1}}, \dots , D_{i_{k_{n}}}$ are all disjoint and contained in 
$Bl(z, s_{i_{k_{n}}}, t s_{i_{k_{n}}})$. By microblossoming and local comparability, for $j = m + 1, \dots, n$ we have
$$
\mu \cup_{j= m+1}^n  D_{i_{k_{j}}}\le \mu Bl(z, s_{i_{k_{n}}}, t s_{i_{k_{n}}}) \le  
K \mu B(z, s_{i_{k_{n}}})\le  K \  C(\mu) \  \mu B(x_{i_{k_{j}}}, s_{i_{k_{n}}}).
$$ 
It follows that 
$$
\sum_{j= m +1}^n \frac{\mu D_{i_{k_{j}}}}{\mu B(x_{i_{k_{j}}},  s_{i_{k_{n}}})}\mathbf{1}_{B(x_{i_{k_{j}}},   
 s_{i_{k_{n}}})} (z)
\le
\frac{ C(\mu) \ \mu Bl(z, s_{i_{k_{n}}}, t s_{i_{k_{n}}})}{\mu B(z, s_{i_{k_{n}}})}
\le C(\mu)  \ K.
 $$
 \end{proof}
 
 Denote by $M_R$ the centered Hardy-Littlewood maximal operator, with the additional restriction 
 that the supremum is taken over radii belonging to the subset $R\subset (0,\infty)$ (cf. \cite[p. 735]{NaTa}).
We mention that under the hypotheses of the next corollary, it is not known whether  
 the centered Hardy-Littlewood maximal operator  $M$ (with no restriction on the radii)  is of weak type (1,1).

\begin{corollary}\label{MR}  Let $(X, d, \mu)$ be a metric measure space, where $\mu$ satisfies a $C(\mu)$ local comparability condition, and let $R:= \{r_n: n\in \mathbb{Z}\}$ be a $T$-lacunary sequence of radii. Suppose there exists a $t > 0$ with $T  t \ge 1$ 
such that 
$\mu$ $(t, K)$-microblossoms boundedly. Then $\|M_R\|_{L^1-L^{1,\infty}} \le (K + 1) \ (C(\mu) \ K + 1)$.
 \end{corollary}
  The proof is standard. We present it  to keep track of the constants.
 
 \begin{proof} Fix $\varepsilon > 0$, let $a > 0$, and let $f\in L^1(\mu)$. 
 For each $x\in \{M_R f > a\}$ select $B(x,r)$ with $r\in R$, such that $a \mu B(x,r) < \int_{B(x,r)} |f|$.
 Then  the collection of ``small" balls $\{ B(x, tr) : x\in  \{M_R f > a\}\}$ is a cover of $\{M_R f > a\}$.
 By the $\tau$-smoothness of $\mu$, there is a finite subcollection 
 $\{B(x_i, t s_i): s_i \in R, 1 \le i \le M\}$ 
of balls with positive measure, ordered by non-increasing radii,  such that
$$
(1 - \varepsilon) \mu \{M_R f > a\} 
\le
(1 - \varepsilon) \mu \cup \{ B(x, tr) : x\in  \{M_R f > a\}\}
< 
\mu \cup_{i= 1}^M B(x_i, t s_i).
$$

Next, let $\{B(x_{i_1}, s_{i_1}), \dots, B(x_{i_N}, s_{i_N})\}$ be the subcollection
given by  the Stein-Str\"omberg covering theorem for sparse radii. Then we have
$$
\mu \cup_{i= 1}^M B(x_i, t s_i)
\le 
(K + 1) \mu \cup_{j=1}^N   B(x_{i_j}, t s_{i_j})
=
(K + 1) \sum_{j=1}^N \mu D_{i_j} 
$$
$$
=
(K + 1) \sum_{j=1}^N \frac{\mu D_{i_j} }{\mu B(x_{i_j}, s_{i_j})}\int \mathbf{1}_{B(x_{i_j}, s_{i_j})}
\le
(K + 1)  \frac{1}{a }\int  |f| \sum_{j=1}^N \frac{\mu D_{i_j} }{\mu B(x_{i_j}, s_{i_j})}\mathbf{1}_{B(x_{i_j}, s_{i_j})}
$$
$$\le (K + 1) \ (C(\mu) \ K + 1) \frac{1}{a }\int  |f| .
$$
 \end{proof}
 
 In the specific case of $d$-dimensional Lebesgue measure $\lambda^d$,  $C(\lambda^d) = 1$.
Choosing
 $t = 1/d$
 and $T = d$,
 $K$ above can be taken to be $e^2$, so the constant obtained
 is $(e^2+1)^2$, which is worse than the constant $(e^2 + 1) (e + 1)$ yielded by the
Stein-Str\"omberg argument. This discrepancy is due
to the fact  that our definition of microblossoming
uses the uncentered blossom instead of the blossom, so from the assumption $\mu (Blu \left(x, r,  t r \right)) \le K \mu B(x,r)$
we  get the same bound $\mu (Bl \left(x, r,  t r \right)) \le K \mu B(x,r)$  for the potentially smaller
centered blossom. Of course, we could strengthen the definition, using blossoms,  to obtain the same 
constant as in the Stein-Str\"omberg proof, but in the case of Lebesgue measure
we prefer to consider it separately, using different values of $(t, K)$ to lower the known bounds. We do this in the next section.

  While Corollary \ref{MR}  follows from the
  proof of the Stein-Str\"omberg covering theorem, it was not stated there but in \cite[Lemma 4]{MeSo} for Lebesgue
measure,
  and in the microdoubling case, in \cite[Corollary 1.2]{NaTa}. A source of  interest for this result comes from the fact that under
$(t,K)$-microblossoming, the maximal operator defined by a $(1 +  t)$-lacunary set of radii $R$ 
is controlled by the sum of $N$ maximal operators with lacunarity $1/t$, where $N$ is the
least integer such that $(1 + t)^N \ge 1/t$. Thus, the bound 
$\|M_R\|_{L^1-L^{1,\infty}} \le  N (K + 1) \ (C(\mu) \ K + 1)$ follows. Under the
additional assumption of $(t, K^{1/2})$-microdoubling, the maximal operator defined by
taking suprema of radii in $[a, (1 + t) a)$ is controlled by $ K^{1/2}$ times the averaging operator 
of radius $(1 + t) a$.
Putting these estimates together, and using the better bound for 
$\mu Bl(x, r, tr)\le  K^{1/2} \mu B(x,R)$, the following result due to  Naor and Tao (cf. \cite[Corollary 1.2]{NaTa})  is obtained. 
Of course, in this case $\mu$ is doubling and $X$, geometrically doubling.

\begin{corollary}  Let $(X, d, \mu)$ be a metric measure space, where $\mu$ satisfies a $C(\mu)$ local comparability condition
and is $(t, K^{1/2})$-microdoubling.
If $N$ is the
least integer such that $(1 + t)^N \ge 1/t$, then
$$\|M\|_{L^1-L^{1,\infty}} \le N \ K^{1/2} \ (K + 1) \ (C(\mu) \ K^{1/2} + 1).$$
 \end{corollary}

This shows that the Stein-Str\"omberg covering theorem for sparse radii in metric spaces suffices to prove
the Naor-Tao bounds, but no greater generality is achieved in either the spaces or the
measures, since microdoubling 
is used in the
last step. A second approach, which yields a slightly more general version of the result  and gives 
 better constants, consists
 in going
back to the original Stein-Str\"omberg argument. Recall that when defining   $(t,K_1)$-microblossoming, 
we set $0 < t <  1$ and  $K_1\ge 1$. In the proof of the next result $K:= \max\{K_1, e\}$ 
 is used to determine the size of the
steps. For convenience we take $K \ge e$, but $e$
is just one possible choice.
 Note that the condition on $ mri(r, T)$ below entails that $\mu$ is doubling on its
support, and hence   $supp(\mu)$ is geometrically doubling. 

\begin{theorem}\label{StSt2} {\bf Stein-Str\"omberg covering theorem for bounded radii.} Let $(X, d, \mu)$ be a metric measure space
such that  $\mu$ satisfies a $C(\mu)$ local comparability condition, and
 is  $(t,K_1)$-microblossoming. Set $K = \max\{K_1 , e\}$.  Let $ r > 0$, and suppose there exists a  $T > 1$ such that
$K_2:= mri(r,T) <\infty$.
Let $\{B(x_i, s_i):  r \le s_i <  T r, 1 \le i \le M\}$ be a finite collection
of balls with positive measure, given in any order, and 
let $D_1 = B(x_1, t s_1), \dots,  D_{M} = B(x_{M}, t s_{M})\setminus \cup_1^{M-1} B(x_{i}, t s_{i})$ be 
the disjointifications of the $t$-reduced balls. Then 
 \begin{equation}\label{sum}
\sum_{i=1}^M\frac{\mu D_i}{\mu B(x_{i}, s_{i})}\mathbf{1}_{B(x_{i}, s_{i})}
\le C(\mu) \ K_1 K \left(2 + \frac{\log K_2}{\log K}\right).
 \end{equation}
 \end{theorem}
 
 Since the big $d\log d$ part in the estimates for the maximal operator (in $\mathbb{R}^d$ with Lebesgue measure) 
comes
 from this case, which does not require any particular ordering nor any  choice  of balls, it is natural
 to enquire whether some additional selection process can lead to an improvement
 in the bounds. In general metric spaces this cannot be done, by \cite[Theorem 1.4]{NaTa}, 
but it might be possible in $\mathbb{R}^d$.  However,  I have not been able to find such a new selection argument.

 In the statement
above, $T$ is not assumed to be close to 1, and in fact it could be much larger than 2
(recall Example \ref{macrodoub}). 
 From the viewpoint of the
proof, the difference between $T >> 2$ and the assumption of $t$-microdoubling lies in the fact that the size of the steps will vary depending on the growth of balls,
rather than having increments given by the constant factor $1 + t$ at every step. But the total number of steps will
be determined by $K$ and  $K_2$, not by whether the factors are all equal to $1 + t$ or not.
 
 \begin{proof} Suppose 
 $$\sum_{i=1}^M\frac{\mu D_i}{\mu B(x_{i}, s_{i})}\mathbf{1}_{B(x_{i}, s_{i})} (y) > 0.
 $$
 Let $s = \min \{s_i: 1 \le i \le M \mbox{ and } y\in B(x_{i}, s_{i})\}$.  Then
 $r \le s < Tr$. Select 
 $$h_1 = \sup \{h > 0 : \mu B(y , (1 + h) s) 
 \le K \mu B^{cl} (y ,s) \mbox{ \ and \ }   (1 + h) s \le  T r\}.
 $$
 This is always possible since $ \lim_{h\downarrow 0}\mu B(y , (1 + h) s) 
 = \mu B^{cl} (y ,s) $.
Now either $(1 + h_1) s =  T r$, in which case the process finishes in one step, and then it could happen
 that
 $\mu B^{cl}(y , (1 + h_1) s) 
 < K \mu B^{cl} (y ,s)$, or $(1 + h_1) s <  T r$, in which case  $\mu B(y , (1 + h_1) s) 
 \le K \mu B^{cl} (y ,s) \le \mu B^{cl}(y , (1 + h_1) s)$  (the last inequality must hold, since otherwise we would be able to select a larger
 value for $h_1$).

 If $h_2, \dots, h_m$ have been chosen, let 
$$h_{m + 1}
 :=  
 \sup \{h > 0 : \mu B(y , s (1 + h)\Pi_{i=1}^m (1 + h_i) \le K \mu B^{cl} (y , s \Pi_{i=1}^m (1 + h_i) )
 \mbox{ \ and \ } 
 $$
 $$
s  (1 + h)\Pi_{i=1}^m (1 + h_i)  \le T r\}.$$
 Since $\mu  B (y , Tr)  < \infty$, the process stops after a finite number of steps, so
  there is an $N \ge 0$ (assigning value 1 to the empty product) such that $ s \Pi_{i=1}^{N + 1} (1 + h_i) = T r$ and
 $$
 \mu  B^{cl} (y , s \Pi_{i=1}^N (1 + h_i) ) \le \mu  B (y , Tr) \le K \mu  B^{cl} (y , s \Pi_{i=1}^N (1 + h_i) ).
 $$
 To estimate $N$, note that since $r \le s$, 
 $$
 \mu B (y , T r) \le K_2 \mu  B (y , s) \le \frac{K_2}{K}   \mu  B^{cl} (y ,(1 + h_1) s) 
 $$
 $$
 \le \dots
 \le   \frac{K_2}{K^N} \mu  B^{cl} (y , s \Pi_{i=1}^N (1 + h_i) ) 
  \le   \frac{K_2}{K^N} \mu  B (y , T r). 
  $$
  Hence $K^N \le K_2$ and thus $N\le \log K_2/ \log K$.

The remaining part of the argument is a variant of what was done in   
Stein-Str\"omberg  for sparse radii, when considering the contribution of balls with the
same radius as the smallest ball. Here we arrange the balls containing $y$  into $N + 2$ ``scales"
(instead of just one) 
depending on whether their radii $R$ are equal to $s$, or $s \Pi_{i=1}^m (1 + h_i) < R \le s \Pi_{i=1}^{m +1} (1 + h_i) $,
or $ s \Pi_{i=1}^N (1 + h_i)  < R \le T r$.

For the first scale, consider all balls 
$B(x_{i_{1,1}}, s), \dots, B(x_{i_{1,k_1}}, s)$ containing $y$. Since for $1\le j \le k_1$, 
$x_{i_{1,j}}  \in B(y, s)$, it follows that the disjoint sets $D_{i_{1,j}}$  are all contained in 
  $Bl(y, s, t s).$
By microblossoming and local comparability we have, for $j = 1, \dots, k_1$,
$$
\mu \cup_{j= 1}^{k_1}  D_{i_{1,j}}\le \mu Bl(y, s, t s) \le  
K_1 \mu B(y, s)\le  K_1 \  C(\mu) \  \mu B(x_{i_{1,j}}, s),
$$ 
so
$$
\sum_{j= 1}^{k_1} \frac{\mu D_{i_{1,j}}}{\mu B(x_{i_{1,j}}, s_{i_{1,j}})}\mathbf{1}_{B(x_{i_{1,j}}, s_{i_{1,j}})} (y)
\le
\frac{ C(\mu) \ \mu Bl(y, s, t s)}{\mu B(y, s)}
\le C(\mu)  \ K_1.
 $$
 
 The contributions of all the other scales are estimated
in the same way as the second one, which is presented next.
Again, consider all balls 
$B(x_{i_{2,1}}, s_{i_{2,1}}), \dots, B(x_{i_{2,k_2}}, s_{i_{2,k_2}})$ containing $y$ and
with radii $s_{i_{2,j}}$ in the interval 
 $(s, (1 + h_1) s]$. Then all the sets  $D_{i_{2,j}}$  are contained in 
  $$
  Bl(y, (1 + h_1) s, t (1 + h_1) s).
  $$  
  Using microblossoming, the choice
of $h_1$,   and 
the local comparability of $\mu$, for $j = 1, \dots, k_2$ we have
 \begin{equation}\label{minus}
\mu \cup_{j= 1}^{k_2}  D_{i_{2,j}}
\le \mu Bl(y, (1 + h_1) s, t (1 + h_1) s) 
  \end{equation}
$$
\le  
K_1 \mu B (y, (1 + h_1) s)
\le  
K_1 K  \mu B^{cl}(y, s)
\le 
K_1 K \  C(\mu) \  \mu B(x_{i_{2,j}}, s_{i_{2,j}}),
$$ 
so
$$
\sum_{j= 1}^{k_2}  \frac{\mu D_{i_{2,j}}}{\mu B(x_{i_{2,j}}, s_{i_{2,j}})}\mathbf{1}_{B(x_{i_{2,j}}, s_{i_{2,j}})} (y)
\le
\frac{ C(\mu) \  \mu Bl(y, (1 + h_1) s, t (1 + h_1) s)}{\mu B^{cl}(y, s)}
\le C(\mu)  \ K_1 K.
 $$
 Adding up over the $N + 2$ scales  we get (\ref{sum}).
  \end{proof}
  
Next we put together the two parts of the Stein-Str\"omberg covering theorem. This helps
to see why the original argument gives  better bounds than domination by several
sparse operators.

\begin{theorem}\label{StSt3} {\bf Stein-Str\"omberg covering theorem.} Let $(X, d, \mu)$ be a metric measure space, where $\mu$ satisfies a $C(\mu)$ local comparability condition, and
 is  $(t,K_1)$-microblossoming. Set $K = \max\{K_1 , e\}$, and suppose 
$K_2:= \sup_{r > 0} mri(r,1/t) <\infty$.
  Let $\{B(x_i, s_i): s_i \in R, 1 \le i \le M\}$ be  a finite collection
of balls with positive measure, ordered by non-increasing radii, and let $U:= \cup_{i = 1}^M   B(x_i, t s_i)$. Then there exists a
subcollection 
$\{B(x_{i_1}, s_{i_1}), \dots, B(x_{i_N}, s_{i_N})\},
$
 such that,   
 denoting by $D_{i_1} = B(x_{i_1}, t s_{i_1}), \dots,  D_{i_N} = B(x_{i_N}, t s_{i_N})
 \setminus \cup_1^{N-1} B(x_{i_j}, t s_{i_j})$, we have
\begin{equation}\label{set2}
\mu U \le (K_1 + 1) \mu \cup_{j=1}^N   B(x_{i_j}, t s_{i_j}),
\end{equation} and
\begin{equation} \label{bound2}
\sum_{j=1}^N \frac{\mu D_{i_j}}{\mu B(x_{i_j}, s_{i_j})}\mathbf{1}_{B(x_{i_j}, s_{i_j})}
\le 1 +  C(\mu) \ K_1 K \left(2 + \frac{\log K_2}{\log K}\right).
\end{equation}
 \end{theorem}
 
  \begin{proof} The selection process is the same as in the proof of Theorem \ref{StSt1},
  yielding the desired subcollection,  with (\ref{set2}) being the same as (\ref{set}). 
  As for the right hand side of (\ref{bound2}) the 1 comes from the contribution of balls
  with very large radii, as in (\ref{bound}), while 
  $C(\mu) \ K_1 K \left(2 + \frac{\log K_2}{\log K}\right)$
  is the bound from (\ref{sum}).
 \end{proof} 

The same argument given for Corollary \ref{MR} now yields
  
 \begin{corollary}\label{M} Under the assumptions and with the notation of the preceding result, 
the centered maximal function satisfies the weak type (1,1) bound  
$$\|M\|_{L^1-L^{1,\infty}} \le (K_1 + 1) \left( 1 +  C(\mu) \ K_1 K \left(2 + \frac{\log K_2}{\log K}\right)\right).$$
 \end{corollary}

 For Lebesgue measure on $\mathbb{R}^d$,
with balls defined by an arbitrary norm and $t = d^{-1}$, this is worse (by a factor of $e^2$) than the bound $(1 + e^2) (1 + o(1)) e^2 d \log d$ 
obtained by Stein and Str\"omberg.

Regarding lower bounds,
currently it is known that for the centered maximal function defined using  $\ell^\infty$-balls (cubes) 
 the  numbers  $\|M\|_{L^1-L^{1,\infty}}$ diverge to
infinity (cf. \cite{A}) 
 at a rate at least $O(d^{1/4})$  (cf.  \cite{IaSt}). No information is available for other balls. In particular, the
question   (asked by Stein and Str\"omberg) as to whether or not the constants $\|M\|_{L^1-L^{1,\infty}}$ diverge to
infinity with $d$, for  euclidean balls, remains open.

\section{Sharpening the bounds for Lebesgue measure}

Here we revisit the  original case studied by 
Stein and Str\"omberg, Lebesgue measure  $\lambda^d$ on $\mathbb{R}^d$, with metric
(and hence, with maximal function) defined by an arbitrary norm.
 Since  $\lambda^d$ is $(t, (1 + t)^d)$-microdoubling for
every $t > 0$, values of $t\ne 1/d$ can be used to obtain  improvements on the size of the constants.

\begin{theorem}\label{SSLebesgue}  Consider $\mathbb{R}^d$ with Lebesgue measure
$\lambda^d$ and balls defined by an arbitrary norm. Let $R:= \{r_n: n\in \mathbb{Z}\}$ be a $d$-lacunary sequence of radii, and let $M_R$ be 
the corresponding (sparsified) Hardy-Littlewood
maximal operator. Then $\|M_R\|_{L^1-L^{1,\infty}} \le (e^{1/d} + 1) (1 + 2  e^{1/d})$.
Furthermore,  if the maximal function is defined using the $\ell_\infty$-norm, so balls
are cubes with sides perpendicular to the coordinate axes, then $\|M_R\|_{L^1-L^{1,\infty}} \le 6.$ 
\end{theorem}

As we noted above, using the original argument from
\cite{StSt} one obtains  $\|M_R\|_{L^1-L^{1,\infty}} \le (e^2 + 1) (e + 1)$.

 \begin{proof} Suppose, for simplicity in the writing, that $r_{n + 1} = d r_n$
(the case  $r_{n + 1} \ge  d r_n$ is proven in the same way).  We  apply the
 Stein Str\"omberg selection process  with $t =1/d^2$ and microdoubling constant $K = (1 + 1/d^2)^{d}
 < e^{1/d}$. 
As before, given $0 \le f \in L^1$ and $a >0$, 
we cover the  level set  $\{M_R f > a\}$  almost completely,  by a finite collection of  ``small" balls
 $\{B(x_i, t s_i): s_i \in R, 1 \le i \le M\}$  ordered by non-increasing radii,
 and such that $a \mu B(x_i,  s_i) < \int_{B(x_i,  s_i)} f$.
 From this collection  we extract a subcollection 
$\{B(x_{i_1}, t s_{i_1}), \dots, B(x_{i_N}, t s_{i_N})\}$
 satisfying 
$$
\mu \cup_{i= 1}^M B(x_i, t s_i)
\le 
(e^{1/d} + 1) \mu \cup_{j=1}^N   B(x_{i_j}, t s_{i_j})
=
(e^{1/d} + 1) \sum_{j=1}^N \mu D_{i_j}. 
$$
Next, we obtain the bound
$$
\sum_{j=1}^N \frac{\mu D_{i_j}}{\mu B(x_{i_j}, s_{i_j})}\mathbf{1}_{B(x_{i_j}, s_{i_j})}
\le 2  e^{1/d} + 1,
$$
by considering $z$ such that
 $\sum_{j=1}^N \frac{\mu D_{i_j}}{\mu B(x_{i_j}, s_{i_j})}\mathbf{1}_{B(x_{i_j}, s_{i_j})}(z) > 0$.
Select the ball $B$   with largest index that contains $z$. Since $B$ belongs to the subcollection
obtained by the Stein-Str\"omberg method, all balls containing $z$ and  with radii
$\ge d^2 r(B)$ (where $r(B)$ denotes the radius of $B$), contribute at most 1 to the sum. 
Next we have to consider two more scales, all the balls with radius $r(B)$,
and all the balls with radius $d r(B)$. By the usual argument (as in the proof of Theorem \ref{StSt1})  each of these scales contributes
at most $e^{1/d}$ to the sum, so $\|M_R\|_{L^1-L^{1,\infty}} \le (e^{1/d} + 1) (1 + 2  e^{1/d})$ follows.
 The result for
 cubes is obtained by letting $d\to\infty$, since in this case it is known that 
 the weak type (1,1) norms increase with the dimension (cf. \cite[Theorem 2]{AV}).
  \end{proof}

\begin{theorem}\label{SS2Lebesgue}  Consider $\mathbb{R}^d$ with Lebesgue measure
$\lambda^d$ and balls defined by an arbitrary norm. If  $\varepsilon > 0$, 
 then  
$\|M\|_{L^1-L^{1,\infty}} \le (2 +  3 \varepsilon) d \log d$  for all $d = d(\varepsilon)$ sufficiently large.
\end{theorem}

 The bound from the proof of \cite[Theorem 1]{StSt} is 
 $\|M\|_{L^1-L^{1,\infty}} \le  e^2 (e^2 + 1) (1 + o(1)) d \log d.$

 \begin{proof} Fix $\varepsilon \in (0,1)$.  Since $(1 + d^{-1 - \varepsilon})^d = 1 + d^{ - \varepsilon}
 + O(d^{ -2  \varepsilon})$, it follows that $\lambda^d$ is $(d^{-1  - \varepsilon},  1 + d^{ - \varepsilon}
 + O(d^{ -2  \varepsilon}))$-microdoubling. Note that if a ball $B$ contains the center of a second  ball of
 radius $1$, and the latter ball is contained in $(1 + d^{-1 - \varepsilon}) B$, then the radius $r_B$
 of $B$ must satisfy $r_B \ge  d^{ 1 +  \varepsilon}$. Let  $L$ be any natural number such that 
 $(1 + d^{-1 - \varepsilon})^L \ge  d^{1 +  \varepsilon}$. Taking logarithms
 to estimate $L$, 
 and using $\log(1 + x) > x - x^2$ for $x$ sufficiently close to $0$, we see that it is
 enough, for the preceding inequality to hold,  to choose $L$ satisfying $L (d^{-1 - \varepsilon} - d^{-2 - 2\varepsilon})
 \ge (1 + \varepsilon) \log d$, or, $L \ge (1 + o(d^{-1}))(1 + \varepsilon) d^{1 + \varepsilon} \log d$.  
For the least such integer we will have 
$$
L \le 1 + (1 + o(d^{-1}))(1 + \varepsilon) d^{1 + \varepsilon} \log d.
$$

 Again we  apply the
 Stein Str\"omberg selection process  with $t = d^{-1 - \varepsilon}$,  
  covering a given  level set  $\{M f > a\}$  almost completely (up to a small $\delta > 0$) by a finite collection of  small balls
 $\{B(x_i, t s_i): s_i \in R, 1 \le i \le k\}$  ordered by non-increasing radii,
 and such that $a \mu B(x_i, t s_i) < \int_{B(x_i, t s_i)} |f|$.
 Using the Stein Str\"omberg algorithm,  we extract a subcollection 
$$
\{B(x_{i_1}, t s_{i_1}), \dots, B(x_{i_N}, t s_{i_N})\}
$$
 satisfying 
\begin{equation} \label{SSsum1}
(1 - \delta) \mu \{M f > a\}
\le 
(2 + d^{ - \varepsilon}
 + O(d^{ -2  \varepsilon})) \sum_{j=1}^N \mu D_{i_j},
 \end{equation}
 where the sets $D_{i_j}$ denote the disjointifications determined by the above subcollection.
To sharpen the usual uniform bound for 
\begin{equation*}
\sum_{j=1}^N \frac{\mu D_{i_j}}{\mu B(x_{i_j}, s_{i_j})}\mathbf{1}_{B(x_{i_j}, s_{i_j})},
\end{equation*}
 we use the fact that the  sets $D_i$ are disjoint
across different steps, and not just within the same step. More precisely, 
let $z$ satisfy
\begin{equation} \label{SSsum}
 \sum_{j=1}^N \frac{\mu D_{i_j}}{\mu B(x_{i_j}, s_{i_j})}\mathbf{1}_{B(x_{i_j}, s_{i_j})}(z) > 0.
\end{equation}
Select the ball $B$   with largest index that contains $z$. Since $B$ belongs to the subcollection
obtained by the Stein-Str\"omberg method, all balls containing $z$ and  with radii
$\ge d^{ 1 +  \varepsilon} r(B)$  contribute at most 1 to the sum. 
Next  we consider the first two scales, since for all the others, the argument
is the same as for the second. 

Take  all the  balls with radii equal to  $r_B$. 
In order to bound (\ref{SSsum}) from above, we suppose that $(1 + d^{ - 1 -  \varepsilon})  B$
is completely filled up with the sets $D_i$ associated to balls with radii $r_B$, and hence,
no $D_j$ associated to a ball with larger radius intersects $(1 + d^{  -1  -  \varepsilon})  B$.
When we consider the sum (\ref{SSsum}), but just for the balls with radius $r_B$,  we obtain
the upper bound $(1 + d^{ - 1 -  \varepsilon})^d$. For the second level, we consider all balls in the
subcollection   with radii in $(r_B, (1 + d^{ - 1 -  \varepsilon}) r_B]$, and
as before, we suppose that $(1 + d^{ - 1 -  \varepsilon})^2 B \setminus (1 + d^{ - 1 -  \varepsilon})  B$
is completely filled up with the sets $D_j$ associated to these balls. The estimate we obtain
for this second level is $(1 + d^{ - 1 -  \varepsilon})^d  - 1 =  d^{ - \varepsilon}
 + O(d^{ -2  \varepsilon})$. For balls with radii in $((1 + d^{ - 1 -  \varepsilon})^k  r_B, 
 (1 + d^{ - 1 -  \varepsilon})^{k + 1} r_B]$, $0 \le k < L$, we use the same estimate.
 Adding up over all scales we obtain
 $$
 \sum_{j=1}^N \frac{\mu D_{i_j}}{\mu B(x_{i_j}, s_{i_j})}\mathbf{1}_{B(x_{i_j}, s_{i_j})}(z) 
 \le
 1 + 1 + d^{ - \varepsilon}
 + O(d^{ -2  \varepsilon})
 $$
 $$
  + (1 +  (d^{ - \varepsilon}
 + O(d^{ -2  \varepsilon})) (1 + o(d^{-1}))(1 + \varepsilon) d^{1 + \varepsilon} \log d)
 \le  (1 + O(d^{-\varepsilon}))(1 + \varepsilon) d \log d.
 $$
 Multiplying this bound with the bound from (\ref{SSsum1})
 and adding an $\varepsilon$ to absorb the big Oh terms, 
 for $d$ large enough we obtain 
$\|M\|_{L^1-L^{1,\infty}} \le (2 +  3 \varepsilon) d \log d$.  
  \end{proof}

\end{document}